\pdfoutput=1
\RequirePackage{ifpdf}
\ifpdf % We~are running pdfTeX in pdf mode
\documentclass[pdftex]{sigma}
\else
\documentclass{sigma}
\fi

\numberwithin{equation}{section}

\newtheorem{Theorem}{Theorem}[section]
\newtheorem*{Theorem*}{Theorem}
\newtheorem{Corollary}[Theorem]{Corollary}

\newtheorem{Proposition}[Theorem]{Proposition}
 { \theoremstyle{definition}

 }

\begin{document}
%\allowdisplaybreaks

\newcommand{\arXivNumber}{2401.03600}

\renewcommand{\PaperNumber}{034}

\FirstPageHeading

\ShortArticleName{A Test of a Conjecture of Cardy}

\ArticleName{A Test of a Conjecture of Cardy}

\Author{Van HIGGS and Doug PICKRELL}
\AuthorNameForHeading{V.~Higgs and D.~Pickrell}
\Address{Mathematics Department, University of Arizona, Tucson AZ 85721, USA}
\Email{\href{mailto:vanh1@arizona.edu}{vanh1@arizona.edu}, \href{mailto:pickrell@arizona.edu}{pickrell@arizona.edu}}

\ArticleDates{Received January 25, 2025, in final form May 04, 2025; Published online May 09, 2025}

\Abstract{In reference to Werner's measure on self-avoiding loops on Riemann surfaces, Cardy conjectured a formula for the measure of all homotopically nontrivial loops in a~finite type annular region with modular parameter~$\rho$. Ang, Remy and Sun have announced a~proof of this conjecture using random conformal geometry. Cardy's formula implies that the measure of the set of homotopically nontrivial loops in the punctured plane which intersect~$S^1$ equals \smash{$\frac{2\pi}{\sqrt{3}}$}. This set is the disjoint union of the set of loops which avoid a ray from the unit circle to infinity and its complement. There is an inclusion/exclusion sum which, in a limit, calculates the measure of the set of loops which avoid a ray. Each term in the sum involves finding the transfinite diameter of a slit domain. This is numerically accessible using the remarkable Schwarz--Christoffel package developed by Driscoll and Trefethen. Our calculations suggest this sum is around $\pi$, consistent with Cardy's formula.}


\Classification{60J67; 30C20; 65E10}

\section{Introduction}

Given a topological space $S$, let ${\rm Comp}(S)$ denote the set of all
compact subsets of $S$ with the Vietoris topology (see \cite{Michael}), and let
\[
{\rm Loop}(S):=\bigl\{\gamma\in {\rm Comp}(S)\mid \gamma\text{ is homeomorphic to } S^1\bigr\}
\]
with the induced topology. Suppose that
for each Riemann surface $S$, $\mu_S$ is a positive Borel measure
on ${\rm Loop}(S)$. This family satisfies conformal restriction if for each conformal embedding
$S_1\to S_2$, the restriction of $\mu_{S_2}$ to ${\rm Loop}(S_1)$ equals
$\mu_{S_1}$. In \cite{W}, Werner proved the
following remarkable result.

\begin{Theorem}%\label{theorem1}
There exists a nontrivial family of locally finite measures
$\{\mu_S\}$ on self-avoiding loops on Riemann surfaces which
satisfies conformal restriction. This family is unique up to
multiplication by an overall positive constant.
\end{Theorem}

Any self-avoiding loop on a Riemann surface is
contained in a finite type annulus $A$, and in turn $A$ is conformally equivalent to a unique
planar annulus $\{1<\vert z\vert<{\rm e}^{\rho}\}\subset \mathbb C^{\times}$, the punctured plane, for some $\rho>0$. Consequently the family $\{\mu_S\}$ is uniquely determined by
\[\mu_0:=\mu\vert_{{\rm Loop}^1(\mathbb C^{\times})}\] the restriction of $\mu$
to loops in the plane which surround $0$. The measure $\mu_0$ is determined by the following formula
of Werner, which implicitly normalizes the family of measures (see \mbox{\cite[Proposition~3]{W}}).

\begin{Theorem} \label{Werner} Suppose that $U$ and $V$ are simply connected
domains such that $0\in U\subset V\subsetneq \mathbb C$. Then
\[
\mu_0\bigl({\rm Loop}^1(V\setminus\{0\})\setminus {\rm Loop}(U)\bigr)=\log\bigl(\bigl\vert\phi'(0)\bigr\vert\bigr),
\]
where $\phi\colon (U,0)\to(V,0)$ is a conformal isomorphism \big(${\rm Loop}^1$ means we consider loops which surround zero\big).
\end{Theorem}

As in \cite[Section 7]{W}, consider the function
\[F_0(\rho):=\mu_0\bigl({\rm Loop}^1(\{1<\vert z\vert<{\rm e}^{\rho}\}\bigr)\]
the measure of the set of loops in the annulus which surround zero.
Cardy (see \cite{Cardy}) conjectured an exact formula
\begin{equation}\label{cardyformula}
F_0(\rho)=6\pi\frac{\sum_{k\in\mathbb Z}(-1)^{k-1}kq^{3k^2/2-k+1/8}}{\prod_{k=1}^{\infty}\bigl(1-q^k\bigr)},\qquad q=\exp\bigl(-2\pi^2/\rho\bigr).
\end{equation}

Ang, Remy and Sun have announced a proof of Cardy's formula in \cite{ARS}, using completely new ideas from random conformal geometry. Our main purpose is to note a corollary of the Cardy formula, and to use this and Theorem \ref{Werner} to numerically test Cardy's formula.

A straightforward application of Poisson summation shows that Cardy's conjecture implies~that
\begin{equation*}%\label{asymptotics}
F_0(\rho)=\rho-\frac{2\pi}{\sqrt{3}}+o\biggl(\frac{1}{\rho}\biggr) \qquad \text{as} \ \rho\to\infty.
\end{equation*}
The leading term is correct (we will discuss this and how it is related to the normalization of~Werner's measure
in Section \ref{normalization}). It is striking that there is not a log term in the expansion. The~basic question is whether we can use this and the constant term to
test Cardy's conjecture.

Let $\Delta$ ($\Delta^*$) denote the open unit disk centered at zero ($\infty$, respectively). Theorem \ref{Werner} implies
\[\mu_0\bigl(\bigl\{\gamma\in {\rm Loop}^1(({\rm e}^{\rho}\Delta)^{\times})\mid \gamma \not\subset \Delta\bigr\}\bigr)=\rho. \]
Therefore,
\[\rho-F_0(\rho)=\mu_0\bigl(\bigl\{\gamma\in {\rm Loop}^1(({\rm e}^{\rho}\Delta)^{\times})\mid \gamma \not\subset \Delta \ \text{and}\ \gamma \not\subset \{1<|z|<{\rm e}^{\rho}\}\bigr\}\bigr).\]
By taking the limit as $\rho\to\infty$, Cardy's conjecture implies that
\begin{align}
\frac{2 \pi}{\sqrt{3}}&{}=\mu_0\bigl(\bigl\{\gamma\in {\rm Loop}^1(\mathbb C^{\times})\mid \gamma \not\subset \Delta,\gamma \not\subset \Delta^*\bigr\}\bigr)\nonumber\\
&{}=\mu_0\bigl(\bigl\{\gamma\in {\rm Loop}^1(\mathbb C^{\times})\mid \gamma \cap S^1 \ne \phi\bigr\}\bigr) \nonumber\\
&{}=\mu_0\biggl(\bigcup_{\eta}\bigl\{\gamma\in {\rm Loop}^1(\mathbb C^{\times}\setminus \eta)\mid \gamma \not\subset \Delta\bigr\}\biggr),\label{formula1}
\end{align}
where the union is over smooth curves $\eta$ connecting a point in $S^1$ to $\infty$. Together with conformal invariance this implies the following corollary.

\begin{Corollary}\label{cardycor}Assuming Cardy's formula, for any circle $C_r:=\{|z|=r\}$
\[\mu_0\bigl(\bigl\{\gamma\in {\rm Loop}^1(\mathbb C^{\times})\mid \gamma\cap C_r\ne \phi\bigr\}\bigr)=\frac{ 2\pi}{\sqrt{3}}. \]
\end{Corollary}

The fact that this measure is finite suggests that similar measures
involving more general loops in place of circles, and more general free homotopy classes on compact surfaces, are also finite.

Returning to (\ref{formula1}), the union over $\eta$ can be partitioned into two pieces, the set of loops which avoid a ray from a point of the unit
circle to $\infty$, and the complement. Therefore, according to Cardy,
\[\frac{2 \pi}{\sqrt{3}}=\mu_0\biggl(\biggl\{\gamma\in \bigcup_{S^1} {\rm Loop}^1\bigl(\mathbb C^{\times}\setminus \mathbb R_{\ge 1}{\rm e}^{{\rm i}\theta}\bigr)\mid \gamma \not\subset \Delta\biggr\}\biggr)+\]
the measure of the set of loops which do not omit a ray from the circle to $\infty$. The measure of this latter set
is definitely positive, because this set has nonempty interior. It is not clear how to analytically evaluate this latter measure. We will mainly focus on Cardy's assertion that
\begin{gather}%\label{formula2}
\frac{2 \pi}{\sqrt{3}}>\mu_0\biggl(\biggl\{\gamma\in \bigcup_{S^1} {\rm Loop}^1\bigl(\mathbb C^{\times}\setminus \mathbb R_{\ge 1}{\rm e}^{{\rm i}\theta}\bigr)\mid\gamma \not\subset \Delta\biggr\}\biggr)
\nonumber \\ \hphantom{\frac{2 \pi}{\sqrt{3}}>}{}
=\lim_{n\to\infty}
\mu_0\biggl(\biggl\{\gamma\in \bigcup_{\{{\rm e}^{{\rm i}\theta}:{\rm e}^{{\rm i}n\theta}=1 \}} {\rm Loop}^1\bigl(\mathbb C^{\times}\setminus \mathbb R_{\ge 1}{\rm e}^{{\rm i}\theta}\bigr)\mid\gamma \not\subset \Delta\biggr\}\biggr).\label{formula3}
\end{gather}

\subsection{Plan of the paper}

In Section \ref{origin}, as a service to mathematicians, we will attempt to explain the origins of Cardy's conjecture, from a physics perspective. Readers should consult the remarkable paper \cite{ARS} for a~novel mathematical perspective on Cardy's conjecture and much more.

In Section \ref{I/E}, we discuss the problem of calculating
\begin{equation*}%\label{fixedn}
\mu_0\biggl(\biggl\{\gamma\in \bigcup_{\{{\rm e}^{{\rm i}\theta}:{\rm e}^{{\rm i}n\theta}=1\}} {\rm Loop}^1\bigl(\mathbb C^{\times}\setminus \mathbb R_{\ge 1}{\rm e}^{{\rm i}\theta}\bigr)\mid\gamma \not\subset \Delta\biggr\}\biggr)
\end{equation*}
using the inclusion/exclusion principle. This leads to the problem of finding a formula for the transfinite diameter
for a slit domain. If $m$ is the number of slits, the cases $m=1,2$ are analytically tractable, but this is generally not so for $m\ge 3$.

In Section \ref{numerical}, we briefly report on our numerical work. Using the SC package of Driscoll and Trefethen (see \cite{DT}), we have found that the limit (\ref{formula3}) appears to be around $\pi$, consistent with Corollary \ref{cardycor}. This involves a long alternating sum that involves large numbers, hence we cannot make a definitive claim about the precise value of (\ref{formula3}).

In Section \ref{Sigma}, we will discuss an alternate way of thinking about the inequalities
\[\mu_0\bigl(\bigl\{\gamma\in {\rm Loop}^1(\mathbb C^{\times})\mid \gamma \cap S^1 \ne \phi\bigr\}\bigr)<\frac{ 2\pi}{\sqrt{3}}<\infty. \]
This finiteness turns out to be equivalent to the integrability of a standard height function on (a~completion of) universal Teichm\"uller space.
We also note that a conjecture of the second author in \cite{CP} is incompatible with Cardy's conjecture.

In the final section, we note that various normalizations of Werner's family of measures are all the same.

\section{An origin story for Cardy's conjecture}\label{origin}

This brief section is included simply as a literature guide for mathematicians.

In \cite[Section 7.4.6]{DMS}, there is a short but lucid description of the $O(n)$ loop model. The~standard~$O(n)$ model is the sigma model (on a space time of some dimension) with target $S^{n-1}$. In~two dimensions, the partition function for a cutoff version on a honeycomb lattice has the~form
\[Z_n(K)=\sum_{{\rm loops}} n^{N_L}K^{N_{B}}, \]
where the sum is over (disjoint) self-avoiding loop configurations on the lattice, $K$ is a coupling constant, $N_L$ is the number of loops, and $N_B$ is the number of bonds (i.e., edges, for all of the loops). This expression was arrived at independently in the study of polymers, as explained in Cardy's book~\cite{Cardy1}.
This expression for the partition function makes sense for any real number $n$. For any $-2\le n\le 2$, it is
possible to define a model on the lattice in a local way, so that the loops emerge as nonlocal observables. For this model, the $O(n)$ loop model, it is believed that
\begin{enumerate}\itemsep=0pt
\item[(1)] there exists a continuous phase transition at the critical value
\[K=K_c(n)=\bigl(2+\sqrt{2-n}\bigr)^{-1/2},\]
\item[(2)] there exists a continuum limit, and
\item[(3)] this continuum limit is a conformal field theory, where
(with $n=-2\cos(\pi g)$) the central charge is
\[c_n=1-6\frac{(g-1)^2}{g}. \]
\end{enumerate}
For special values of $n$, it is a minimal model.

Suppose that $n=0$, i.e., $c=0$. In this case, the partition function is trivial, $Z_0=1$. However, the derivative
\begin{equation*}%\label{partitionfn}
\frac{\partial}{\partial n}Z_n|_{n=0}=\sum_{\rm loops} N_L n^{N_L-1}K^{N_{B}}|_{n=0}
\end{equation*}
is a sum over single self-avoiding loops weighted according to their length. Random self-avoiding loops weighted according to their length (in an appropriate way, depending on the lattice) conjecturally converges to Werner's measure (see \cite[Section 7.1]{W}).

In \cite{Cardy}, for the continuum $O(n)$ model, Cardy derives an exact expression for the partition function for an annulus with free boundary conditions. This expression had been derived previously by other methods, by Saleur and Bauer, see \cite{SB}. The derivative with respect to $n$, at~${n=0}$, of this partition function
is the Cardy formula (\ref{cardyformula}). Cardy's conjecture is that this derivative is the Werner measure of homotopically nontrivial loops in the annulus.

Readers should see \cite{ARS} for a contemporary mathematical perspective, including a proof of the character formula.

\section{Applying inclusion/exclusion}\label{I/E}

Riemann mapping (or Schwarz--Christoffel) formulas for slit domains are well-known and deceptively simple (see \cite{DT}).

\begin{Proposition}
Fix angles
\[0\le \phi_1<\phi_2<\dots <\phi_m<2\pi.\]
The unique conformal isomorphism
\[\phi_+\colon \ (\Delta,0) \to \Biggl(\mathbb C \setminus\bigcup_{j=1}^m \mathbb R_{\ge 1}{\rm e}^{{\rm i}\phi_j},0\Biggr)\]
with $\phi_+'(0)>0$,
has the form
\[\phi_+(z)= \rho_0 z \prod_{j=1}^m \biggl(1-\frac{z}{z_j}\biggr)^{-\gamma_j},\]
where $\gamma_j \pi=\phi_j-\phi_{j-1}$, $j=1,\dots ,m-1$, $\sum_{j=1}^m\gamma_j=2$, $z_j\in S^1$,
and
\[\rho_0=\rho_0(\gamma_1,\dots ,\gamma_m)>0 .\]

Similarly, the unique conformal isomorphism
\[\phi_-\colon \ (\Delta^*,\infty) \to \Biggl(\widehat{\mathbb C} \setminus\bigcup_{j=1}^m \mathbb [0,1]{\rm e}^{{\rm i}\phi_j},\infty\Biggr)\]
with positive derivative at $\infty$, has the form
\[\phi_-(z)= \rho_{\infty} z \prod_{j=1}^m \biggl(1-\frac{z_j}{z}\biggr)^{\gamma_j},\]
where $\rho_{\infty}=1/\rho_0$ is the transfinite diameter for the compact set
$\bigcup_{j=1}^m \mathbb [0,1]{\rm e}^{{\rm i}\phi_j}$.
\end{Proposition}

The points $z_j\in S^1$ are usually referred to as accessory parameters. Here is one theoretical approach to finding these points. First, assuming we have the points $z_j$, we find the $m $ points~${z=c_j}$ which map to the points $\exp({\rm i}\phi_j)$. The derivative of $\phi_+$ vanishes at these points, $\phi_+'(c_j)=0$, hence these points $z=c_j$ satisfy
\[1=\sum_{j=1}^m \gamma_j\frac{z}{z-z_j} \]
or
\[\prod_{j=1}^m(z-z_j)=\sum_{j=1}^m \gamma_jz\prod_{i\ne j}(z-z_i) .\]
We can order the solutions so that
\[z_m<c_1<z_1<c_2<z_2<\dots <c_m<z_m\]
in reference to the orientation of the circle. Hence
\[c_j=c_j(z_1,\dots ,z_m).\]
Now we consider the $m$ equations
\[ \rho_0 c_i \prod_{j=1}^m \biggl(1-\frac{c_i}{z_j}\biggr)^{-\gamma_j}=\exp({\rm i}\phi_i),\qquad i=1,\dots ,m.\]
We can fix $z_0=1$ so that there are $m$ unknowns. Note that
in the end
\begin{equation}\label{rhoformula}\rho_0=\prod_{j=1}^m \biggl|1-\frac{c_i}{z_j}\biggr|^{\gamma_j} \end{equation}
for any $i=1,\dots ,m$.

There are alternate expressions for the transfinite diameter $\rho_{\infty}=1/\rho_0$ due to Fekete, Polya, Szego and others, see \cite[Sections~16 and~17]{Hille}. Let $E=\bigcup_{j=1}^m \mathbb [0,1]{\rm e}^{{\rm i}\phi_j}$. By definition
\[\delta_n(E)^{\binom{n}{2}}=\max_{\{w_1,\dots ,w_n\}\subset E} \prod_{1\le i<j\le n} |w_i-w_j| .\]
The basic fact is that $\delta_n \downarrow \rho_{\infty}$ as $n\uparrow \infty$. This (eventually) implies
\[\ln(\rho_0)=\lim_{n\to\infty} \min_{w_i\in E} \frac{1}{\binom{n}{2}} \sum_{1\le i<j\le n} \ln\bigl(|w_i-w_j|^{-1}\bigr) .\]
This limit can also be written as
\[=\int \int \ln\bigl(|z-w|^{-1}\bigr) {\rm d}\nu(z){\rm d}\nu(w) ,\]
where $\nu$ is the electrostatic distribution on $E$.
It is tempting to believe that for slit domains one could work this out explicitly, avoiding
the accessory parameters. This seems to not be the case, as we discuss below.

\begin{Theorem}\label{theorem2} Fix $n$.
\begin{enumerate}
\item[$(a)$] We have
\begin{gather*}
\mu_0\biggl(\bigcup_{\{{\rm e}^{{\rm i}\theta}:{\rm e}^{{\rm i}n\theta}=1\}} {\rm Loop}^1\bigl(\mathbb C^{\times}\setminus \mathbb R_{\le 1}{\rm e}^{{\rm i}\theta}\bigr)\setminus {\rm Loop}(\Delta^{*})\biggr)\\
\qquad{}=\sum_{m=1}^n (-1)^{m-1}\sum_{0\le \phi_1<\phi_2<\dots <\phi_m<2\pi}\ln(\rho_0(\gamma_1,\dots ,\gamma_m)),
\end{gather*}
where $\exp({\rm i}n\phi_j)=1$, and if $\phi_{m+1}=\phi_1$, then $\gamma_i=\phi_{i+1}-\phi_i$, $i=1,\dots ,m$.

\item[$(b)$] For any $1\le i\le m$,
\[\rho_0=\prod_{j=1}^m |z_j-c_i|^{\gamma_j}.\]
Hence also
\[\rho_0=\Biggl(\prod_{j=1}^m a_j^{\gamma_j}\Biggr)^{1/m},\qquad \text{where}\ a_j=a_j(\gamma_1,\dots ,\gamma_m)=\prod_{i=1}^m\biggl|1-\frac{c_i}{z_j}\biggr|,\]
the geometric mean of the distances from $z_j$ to the points $c_i$.
When $m=2$, $a_1(\gamma_1,\gamma_2)=2\gamma_1$ and if all the $\gamma_j=2/m$, then
$a_j(2/m,\dots ,2/m)=m(2/m)=2$.

\item[$(c)$] The sum in $(a)$ can also be written as
\[\sum_{m=1}^n(-1)^{m-1}\sum_{\gamma_1+\dots +\gamma_m=2}|\{\phi_1 \mid 0\le \phi_1<\pi\gamma_m\}|\ln(\rho_0(\gamma_1,\dots ,\gamma_m)),\]
where the $\gamma_j$ are multiples of $\frac 2n$ in the inner sum.
\end{enumerate}
\end{Theorem}

\begin{proof} The formula in (a) follows from the inclusion/exclusion principle, Werner's formula,
and the form of $\phi_+$.

Part (b) follows from the discussion preceding the statement of the theorem, see (\ref{rhoformula}).

(c) Given the $\phi_j$ satisfying
\[0\le \phi_1<\dots <\phi_m<2\pi, \qquad {\rm e}^{{\rm i}n\phi_j}=1, \]
by definition the $\gamma_j$ are multiples of $\frac 2n$ and satisfy $\sum_{j=1}^m\gamma_j=2$. Conversely,
suppose we are given the $\gamma_j$. Then we can solve $\phi_2=\pi\gamma_1+\phi_1$, $\phi_3=\pi(\gamma_1+\gamma_2)+\phi_1$, \dots, and \[\phi_m=\pi(\gamma_1+\dots +\gamma_{m-1})+\phi_1=2\pi-\pi\gamma_m+\phi_1,\]
$\phi_1$ is not determined, but it must satisfy the inequality
\[2\pi-\pi\gamma_m+\phi_1<2\pi \qquad \text{or} \qquad \phi_1<\pi \gamma_m.\]
We have to count the number of possible $\phi_1$, as we have done in the second formula. Suppose that the largest $\phi_1=\frac{2\pi k}{n}$. Then
\[ \frac{2 k}{n}<\gamma_m \qquad \text{or} \qquad k<\frac{n\gamma_m}{2} .\]
This implies (c).
\end{proof}

\subsection{Some examples}

The equations for the accessory parameters can be solved in the cases $m=1,2$, but otherwise they seem quite intractable.

Suppose that the slit domain is $\mathbb C \setminus [0,\infty)$, i.e., $m=1$ and $\phi_1=0$. In this case,
\[\phi_+(z)=4\frac{z}{1+z^2},\]
 which is essentially the Koebe function.

From this, one can derive other explicit examples by using the elementary fact that if $f=z\bigl(1+\sum_{n\ge 1}u_nz^n\bigr)$ is a univalent function on the disk, then
\begin{equation}\label{principle}g(z)=f(z^n)^{1/n}\end{equation} is also univalent function.
Thus using the Koebe function, if the $\phi_j$ are the $m$-th roots of unity, then
\[\phi_+(z)=2^{2/m}z(1+z^m)^{-2/m}, \]
and hence \smash{$\rho_{\infty}=\bigl(\frac{1}{2}\bigr)^{2/m}$}.

Suppose that $m=2$ and $\gamma_1=\gamma$, $\gamma_2=2-\gamma$. In this case, one can
solve the equations of the previous section to find
\begin{equation}\label{m=2}
\rho_0=4\Bigl(\frac{\gamma}{2}\Bigr)^{\frac{\gamma}{2}}\Bigl(1-\frac{\gamma}{2}\Bigr)^{1-\frac{\gamma}{2}},
\end{equation}
which is reminiscent of entropy. Using (\ref{principle}), this leads to an infinite number of other explicit examples. From these examples one can check that $\rho_0$ is not in general a symmetric function of the $\gamma_j$.

The expression in part (b) of Theorem \ref{theorem2} is a generalization
of the entropy like formula~(\ref{m=2}); unfortunately the $a_j$ in part (b) are not easy to find.

Suppose that $m=3$, $\phi_0=0$, $\phi_1=\pi \gamma$, $\phi_2=2\pi-\pi\gamma$, hence $\gamma_1=\gamma_3=\gamma$, $\gamma_2=2-2\gamma$.

In this case, $c_1=1$ and $z_2=-1$. If $z_1=\exp({\rm i} t)$, then using the formula for the derivative of~$\phi_+$,
\[{\rm Re}(c_2)=\cos(t)-\gamma(1+\cos(t)). \]
Now
\[\phi_+(z)=\rho_0 z(1+z)^{-(2-2\gamma)}\Bigl(1-\frac{z}{z_1}\Bigr)^{-\gamma}\Bigl(1-\frac{z}{\overline{z_1}}\Bigr)^{-\gamma} .\]
Using $\phi_+(1)=1$, we obtain (after simplifying)
\[\rho_0=2^{2-\gamma} (1-\cos(t))^{\gamma} .\]
To find $z_1=\exp({\rm i}t)$, we have to solve for $t$ in
\[\phi_+(c_2(t))=\exp({\rm i}\pi \gamma),\]
i.e., we have to solve for $\cos(t)$ using the equation
\[2^{2-\gamma} (1-\cos(t))^{\gamma} c_2\bigl(1+c_2^{-1}\bigr)^{2-2\gamma}\bigl(1-2\cos(t)c_2^{-1}+c_2^{-2}\bigr)^{\gamma}=\exp({\rm i}\pi\gamma). \]
If $x=\cos(t)$, then
\begin{align*}
\rho_0&{}=2^{2-\gamma}(1-x)^{\gamma}=|1+c_2|^{2(1-\gamma)}\bigl|c_2^2-2c_2x+1\bigr|^{\gamma}\\
&{}=2(1+x)\gamma^{\gamma}(1-\gamma)^{1-\gamma}=2^{2-\gamma}(1-x)^{\gamma}.
\end{align*}
The function \smash{$y=\frac{(1-x)^{\gamma}}{1+x}$} decreases monotonically from 1 to 0 as $x$ increases from 0 to 1.
So there is a well-defined inverse.

The function $\gamma \to 2^{\gamma-1}\gamma^{\gamma}(1-\gamma)^{1-\gamma}$, $0\le \gamma\le 1$, starts at 1/2,
decreases to near zero, then increases to 1. In any event it has values between 0 and 1.
We can and do simply write
\[x=x\bigl(2^{\gamma-1}\gamma^{\gamma}(1-\gamma)^{1-\gamma}\bigr).\]

\begin{Proposition}We have
\[\rho_0(\gamma,\gamma,2-2\gamma)=2(1+x)\gamma^{\gamma}(1-\gamma)^{1-\gamma},\]
where
\[x=x\bigl(2^{\gamma-1}\gamma^{\gamma}(1-\gamma)^{1-\gamma}\bigr).\]
\end{Proposition}

This is a formula which is not very enlightening. It illustrates that one is better off using a~numerical method.

\section{Numerically testing Cardy's formula}\label{numerical}

We now want to use the formula in part~(a) of Theorem \ref{theorem2} to test Cardy's formula.
The~alternating sum
\begin{equation*}%\label{keysum}
\sum_{m=1}^n (-1)^{m-1}\sum_{0\le \phi_1<\phi_2<\dots <\phi_m<2\pi}\ln(\rho_0(\gamma_1,\dots ,\gamma_m)),
\end{equation*}
where $\exp({\rm i}n\phi_j)=1$, is an increasing function of $n$, and the question is whether it converges to something
less than \smash{$\frac{2\pi}{\sqrt{3}}$}, consistent with Cardy's conjecture, or if the sequence eventually exceeds
this number.

The number of terms in this sum grows very rapidly, and the individual terms are relatively large.
In the first version of this paper, we proposed a number of approximations (e.g., using an entropy like expression
for $\rho_0$) which alternately suggested the sum converges and diverges. More recently, we discovered the SC package
of Driscoll and Trefethen (see~\cite{DT}). Using this package, we have been able to compute the alternating sum up to $n=22$
using the diskmap method from the SC toolbox. The results of this process for $n=1$ to $n=22$ are shown in Figure~\ref{fig1}.

\begin{figure}[t]\centering
\includegraphics[scale = 0.7]{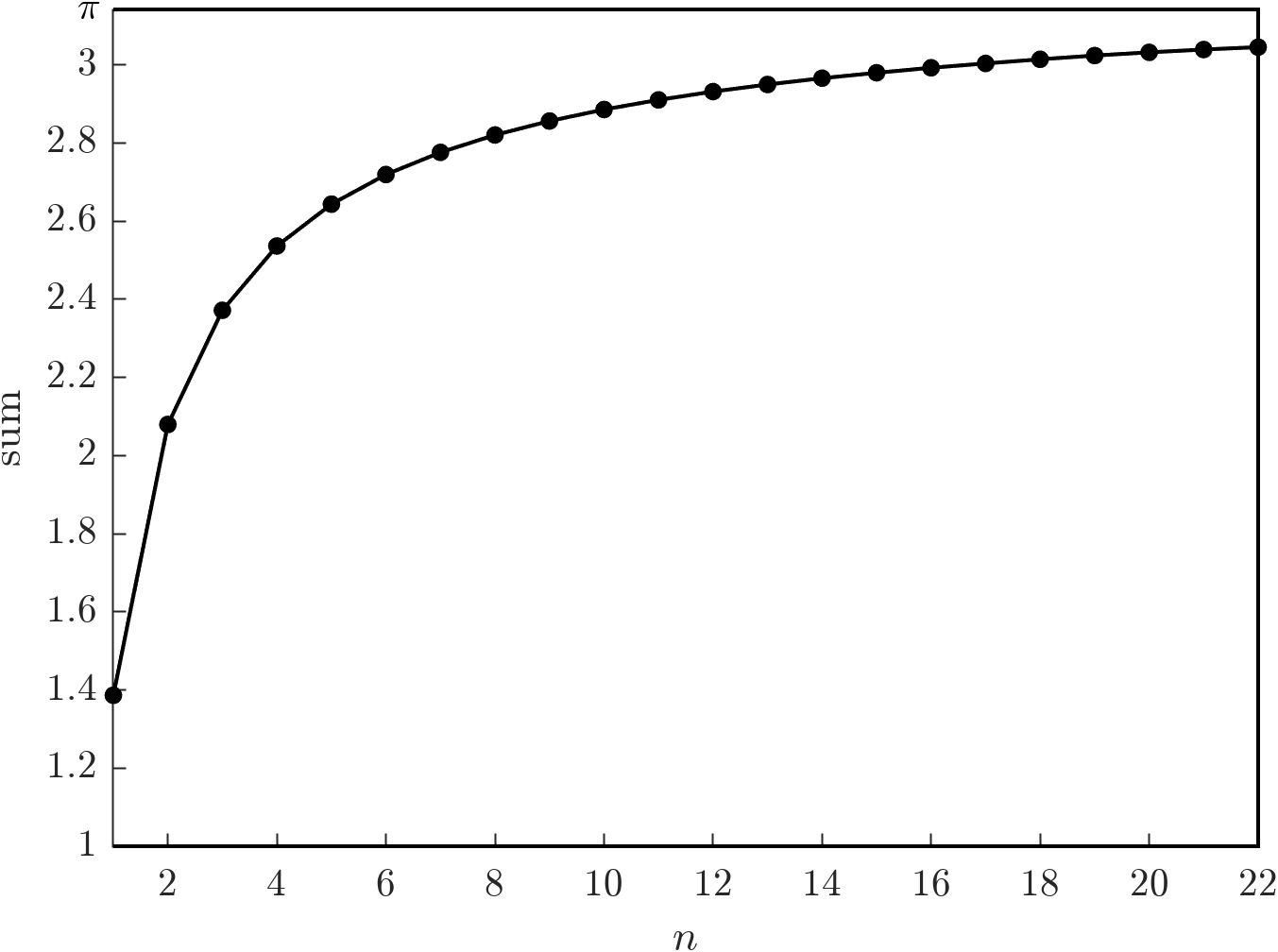}
\caption{The last 5 outputted values are 3.0140, 3.0236, 3.0318, 3.0390, and 3.0451. }\label{fig1}
\end{figure}

Although far from conclusive, it seems plausible that the sum is converging to something approaching $\pi$.

\section{Integrability of a height function}\label{Sigma}

Suppose that $\gamma\in{\rm Loop}^1(\mathbb C^{\times})$. By the
Jordan curve theorem, the complement of $\gamma$ in $\mathbb C
\cup\{\infty\}$ has two connected components, $U_{\pm}$, so that
\[\mathbb C \cup \{\infty\}=U_+ \sqcup \gamma \sqcup U_-,\]
where $0\in U_+$
and $\infty \in U_-$. There are based conformal isomorphisms
\[\phi_+\colon \ (\Delta,0) \to (U_+,0),\qquad \phi_-\colon \ (\Delta^*,\infty) \to (U_-,\infty).\]
The map $\phi_-$ can be uniquely determined by normalizing the
Laurent expansion in $\vert z\vert>1$ to be of the form
\begin{equation*}%\label{phi-}
\phi_-(z)=\rho_{\infty}(\gamma)L(z),\qquad L(z)=z\biggl(1+\sum_{n\ge 1}b_nz^{-n}\biggr),
\end{equation*}
where $\rho_{\infty}(\gamma)>0$ is the
transfinite diameter (see \cite[Sections 16 and 17]{Hille} for numerous formulas for
$\rho_{\infty}$). The map $\phi_+$ can be similarly uniquely
determined by normalizing its Taylor expansion to be of the form
\begin{equation*}%\label{phi+}
\phi_+(z)=\rho_0(\gamma)u(z),\qquad u(z)=z\biggl(1+\sum_{n\ge 1}u_nz^{n}\biggr),
\end{equation*}
where $\rho_0(\gamma)>0$ is called the conformal radius with respect to
$0$. By a theorem of Carath\'{e}odory (see \cite[Theorem 17.5.3]{Hille}),
both $\phi_{\pm}$ extend uniquely to homeomorphisms
of the closures of their domain and target. This
implies that the restrictions $\phi_{\pm}\colon S^1 \to \gamma$ are
topological isomorphisms. Thus there is a well-defined welding map
\begin{equation*}%\label{welding}
W\colon \ {\rm Loop}^1(\mathbb C^{\times}) \to
\bigl\{\sigma\in {\rm Homeo}^+\bigl(S^1\bigr)\mid \sigma=lau\bigr\} \times \mathbb R^+, \qquad \gamma
\mapsto (\sigma(\gamma),\rho_{\infty}(\gamma)),
\end{equation*}
where
\[
\sigma(\gamma,z):=\phi_-^{-1}(\phi_+(z))=lau,\qquad a(\gamma)=\frac{\rho_0(\gamma)}{\rho_{\infty}(\gamma)}
\]
and $l$ is the inverse mapping for $L$. In these `coordinates'
\begin{equation}\label{normalization2}
{\rm d}\mu_0(\gamma)= {\rm d}\nu_0(\sigma) \times \frac{{\rm d}\rho_{\infty}}{\rho_{\infty}},
\end{equation}
where $\nu_0$ is a probability measure, see \cite[Proposition~2.1]{CP} and Section~\ref{normalization} below.

Let $H(\gamma)=-\log(a)\ge 0$. This can be viewed as a height function on universal Teichm\"uller space.
There are various expressions for this function, e.g.,
\begin{align*}
H(\gamma)&{}=\log(\rho_{\infty})-\log(\rho_0)\\
&{}=\log\biggl(1+\sum_{n=1}^{\infty}(n+1)|u_n|^2\biggr)-\log\biggl(1+\sum_{m=1}^{\infty}(m-1)|b_m|^2\biggr)
\end{align*}
and it is inversion invariant (see \cite[Section 1]{CP}).

As observed in \cite[Section 7]{CP},
\begin{equation*}%\label{setinclusions}
{\rm Loop}^1(\{1<\vert z\vert<{\rm e}^{\rho}\})\subset\{1\le \rho_0\le
\rho_{\infty}\le {\rm e}^{\rho}\}\subset {\rm Loop}^1\biggl(\biggl\{\frac14<\vert
z\vert<4{\rm e}^{\rho}\biggr\}\biggr)
\end{equation*}
and as a consequence
\begin{equation*}%\label{intertwine}
F_0(\rho)\le \int_0^{\rho}\nu_0(H\le x){\rm d}x\le
F_0(\log(16)+\rho).
\end{equation*}
Using Werner's formula, this implies
\[ \int_0^{\rho}\nu_0(H\ge x){\rm d}x\le \rho-F_0(\rho) \]
and taking the limit as $\rho\uparrow \infty$
\[ \int H {\rm d}\nu_0 \le \mu_0\bigl(\bigl\{\gamma\mid \gamma \cap S^1\ne \phi\bigr\}\bigr).\]
Thus we have the following assertion.

\begin{Theorem}Assuming Cardy's formula,
\[ \frac{2\pi}{\sqrt{3}}-\ln(16)\le \int H {\rm d}\nu_0 \le \frac{2\pi}{\sqrt{3}}\]
and in particular $H$ is integrable.
\end{Theorem}

In \cite[Section 7]{CP}, the second author conjectured that there is an exact formula
$\nu_0(H\le x)=\exp(-\beta_0/x)$, where $\beta_0=2\pi^2\cdot (1/8)$ (the $1/8$ is a critical exponent from logarithmic conformal field theory). This is not compatible with Cardy's conjecture. The reason is that Cardy's formula for $F_0(\rho)$ does not have a log term in its asymptotic expansion as $\rho\uparrow \infty$, whereas $\int_0^{\rho}\nu(H\le x){\rm d}x$ does have a log term.

\section{Normalizations of Werner's family of measures}\label{normalization}

Werner's family of measures is unique up to a positive constant.
There are several ways to normalize the family. The first is using Werner's formula
Theorem \ref{Werner}. A second is to assume $\nu_0$ in (\ref{normalization2}) is a probability measure (this
is the normalization adopted in \cite{CP}).
A third is to assume that $F_0(\rho)$ is asymptotically $\rho$, which is implicit in Cardy's formula.

\begin{Theorem} These normalizations are all the same.
\end{Theorem}

\begin{proof}(This is a slight modification of the proof of \cite[Proposition 7.5]{CP}.) As in \cite{CP} (and as we assumed in the previous section), assume that $\nu_0$ is a probability
measure. This means that we must insert a constant $c_W$ into Theorem~\ref{Werner} (as we did in \cite{CP}).
Thus if $\gamma$ is
a loop which surrounds $\Delta$,
\[\mu_0\bigl({\rm Loop}^1(U_+\setminus\{0\})\setminus {\rm Loop}(\Delta)\bigr)=c_W
\log(\rho_0(\gamma)).\]
We must show $c_W=1$ and $F_0(\rho)$ is asymptotically
$\rho$.

As in \cite[Proposition 7.5]{CP}, on the one hand
\[{\rm Loop}^1(\{1<\vert z\vert<{\rm e}^{\rho}\})\subset {\rm Loop}^1(\{\vert z\vert<{\rm e}^{\rho}\})\setminus {\rm Loop}(\Delta) .\]
On the other hand,
\[
{\rm Loop}^1(\{1<\vert z\vert<{\rm e}^{\rho}\})\subset\{1\le \rho_0\le \rho_{\infty}\le {\rm e}^{\rho}\}\subset {\rm Loop}^1\biggl(\biggl\{\frac14<\vert
z\vert<4{\rm e}^{\rho}\biggr\}\biggr),
\]
where the last inclusion uses Koebe's quarter theorem. Therefore,
\[F_0(\rho)\le \int_0^{\rho}\nu_0({\rm e}^{-x}\le a\le 1){\rm d}x\le
F_0(\log(16)+\rho) .\]
Since $\nu_0$ is a probability measure, for $\rho\gg 1$
\begin{equation}
\label{ineq}\rho-\ln(16)\le F_0(\rho)\le \rho .\end{equation}
Thus $F_0(\rho)$ is asymptotically $\rho$.

Werner's formula implies
\[
\mu_0\bigl({\rm Loop}^1(\{0<|z|<{\rm e}^{\rho}\})\setminus {\rm Loop}(\Delta)\bigr)=c_W\rho.
\]
This can also be written as
\[\lim_{\epsilon\downarrow 0}\left(F_0(\rho-\ln(\epsilon))-F_0(-\ln(\epsilon))\right). \]
Together with (\ref{ineq}) this implies (for $\rho\gg 1\gg \epsilon$)
\[\rho-\ln(\epsilon)-\ln(16)-(-\ln(\epsilon))\le
c_W\rho \le \rho-\ln(\epsilon)-(-\ln(\epsilon)-\ln(16)) \]
and
\[\rho-\ln(16)\le c_W\rho \le \rho+\ln(16) .\]
This implies $c_W=1$.
\end{proof}

\subsection*{Acknowledgements} We thank Toby Driscoll for help with using the SC package for slit domains, and the referees
for comments which improved the exposition.

%\cite{AKS,IT,KS} - not cited

\pdfbookmark[1]{References}{ref}
\LastPageEnding

\end{document}